\documentclass[reqno,10pt]{amsart}

\usepackage{amsmath,amsfonts,amsthm,amssymb,amscd}
\usepackage{graphicx} 
\usepackage{epsfig}
\usepackage{inputenc}
\usepackage{upref} 
\usepackage{mathptmx}
\usepackage[T1]{fontenc}
\usepackage{bm}
\usepackage{calc}
\usepackage{caption}
\usepackage{subcaption}

\newtheorem{theorem}{Theorem}[section]
\newtheorem{conjt}[theorem]{Conjecture}
\newtheorem{corollary}[theorem]{Corollary}
\newtheorem{lemma}[theorem]{Lemma}
\newtheorem{proposition}[theorem]{Proposition}

\newtheorem{definition}[theorem]{Definition}

\newtheorem{remark}[theorem]{Remark}

\newcommand\bd{\begin{displaymath}}
\newcommand\ed{\end{displaymath}}
\newcommand\be{\begin{equation}}
\newcommand\ee{\end{equation}}
\newcommand\bea{\begin{eqnarray}}
\newcommand\eea{\end{eqnarray}}
\newcommand\bi{\begin{itemize}}
\newcommand\ei{\end{itemize}}
\newcommand\ben{\begin{enumerate}}
\newcommand\een{\end{enumerate}}
\newcommand\bc{\begin{center}}
\newcommand\ec{\end{center}}
\newcommand\ba{\begin{array}}
\newcommand\ea{\end{array}}

\newcommand{\R}{\mathbb{R}}

\title[Tiling Space by n-Hedra]{Surface-area-minimizing $n$-hedral Tiles}

\begin{document}

\author[W. Ghang]{Whan Ghang}
\author[Z. Martin]{Zane Martin}
\author[S. Waruhiu]{Steven Waruhiu}


\begin{abstract}
We provide a list of conjectured surface-area-minimizing $n$-hedral tiles of space for $n$ from 4 to 14, previously known only for $n$ equal to 5 or 6. We find the optimal "orientation-preserving" tetrahedral tile $(n=4)$, and we give a nice proof for the optimal 5-hedron (a triangular prism).
\end{abstract}

\maketitle

\section{Introduction}

For fixed $n$, we seek a unit-volume $n$-hedral tile of space that minimizes surface area. Our Conjecture \ref{best3Dtiles} provides candidates from $n=4$, a certain irregular tetrahedron, to $n\geq 14$, Kelvin's truncated octahedron (see Figs. 1-7). The conjecture is known for $n=6$ and $n=5$. That the cube is the best 6-hedron, tile or not, is well known \cite{ftoth} (See Thm. \ref{Florianpf}). Theorem \ref{existspoly} shows that among convex polyhedra, for fixed $n$, there exists a surface-area-minimizing $n$-hedral tile of space. Section \ref{secprism} gives some properties of prisms and a proof that a certain hexagonal prism is the surface-area-minimizing prism. Theorem \ref{bestfivepoly} gives a nice new proof that a certain triangular prism is the surface-area-minimizing 5-hedron. Theorem \ref{besttetra} proves that a third of a triangular prism is the surface-area-minimizing "orientation-preserving" 4-hedral tile, based on a classification of tetrahedral tiles by Sommerville \cite{somville}. (Unfortunately the regular tetrahedron does not tile space.)

\subsection{Acknowledgements}
This paper is work of the 2012 ``SMALL'' Geometry Group, an undergraduate research group at Williams College continued by Waruhiu. Thanks to our advisor Frank Morgan, for his patience, guidance, and invaluable input. Thanks to Andrew Kelly and Max Engelstein for contributions to the summer work that laid the groundwork for this paper. Thanks to the National Science Foundation for grants to Morgan and the Williams College ``SMALL'' Research Experience for Undergraduates, and to Williams College for additional funding. Additionally thank you to the Mathematical Association of America (MAA), MIT, the University of Chicago, and Williams College for grants to Professor Morgan for funding in support of trips to speak at MathFest 2012 and the Joint Meetings 2013 in San Diego.

\section{Tiling of Space}
\label{space}

We assume that a space-filing polyhedron tiles $\R^3$ with congruent copies of itself and the polyhedra are face-to-face, i.e., that polyhedra meet only along entire faces, entire edges, or at vertices. We have the following conjecture:

\begin{conjt}
\label{best3Dtiles}
For fixed $n$ and unit volume, the following provide the surface-area-minimizing $n$-hedral tiles of $\R^3$ (see Figs. 1-7):
\begin{enumerate}
\item $n=4$: a tetrahedron formed by four isosceles right triangles with two sides of $\sqrt{3}$ and one side of 2. It is also formed by cutting a triangular prism into three congruent tetrahedra;
\item $n=5$: a right equilateral-triangular prism;
\item $n=6$: the cube;
\item $n=7$: a right Cairo or Prismatic pentagonal prism;
\item $n=8$: the gabled rhombohedron described by Goldberg \cite{goldbergocta} as having four pentagonal and four quadrilateral sides and the hexagonal prism;
\item $n=9$: an enneahedron with three non-adjacent congruent square faces and six congruent pentagonal faces;
\item $n=10$ and $11$: a decahedral "barrel" with congruent square bases and eight congruent pentagonal sides;
\item $n=12$: a 12-hedron of Type 12-VIII described by Goldberg \cite{goldbergdodeca} with 20 vertices of degree three and none of degree four (one half the truncated octahedron (10));
\item $n=13$: a 13-hedron of Type 13-IV described by Goldberg \cite{goldberg>12} as cutting a 14 sided hexagonal prism capped on each end by four faces in half;
\item $n \geq 14$: Kelvin's truncated octahedron (\cite{kelvin}, see \cite[pp. 157-171]{morggeo}).

\end{enumerate}
\end{conjt}

\begin{remark}
\emph{Goldberg (\cite[p.231]{goldberg}, see \cite[p. 213]{florian}) conjectured that a surface-area-minimizing $n$-hedron has only vertices of degree three, but it may well not tile. All the vertices of our conjectured polyhedra have degree three.}
\end{remark}

\begin{figure}
	\centering
	\centering
			\includegraphics[scale=0.7]{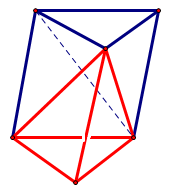}
			\caption{A tetrahedron formed by cutting a triangular prism into three congruent tetrahedra is the conjectured surface-area-minimizing tetrahedral tile.}
			\label{fig:besttetrah}
            \centering
			\includegraphics[scale=0.7]{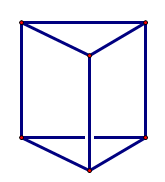}
			\caption{A right equilateral-triangular prism is the surface-area-minimizing 5-hedron.}
	\includegraphics[scale=0.7]{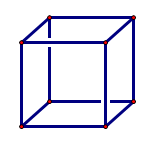}
	\caption{The cube is the surface-area-minimizing 6-hedron.}
	\includegraphics[scale=0.6]{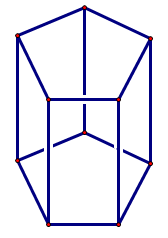}
	\caption{A right Cairo prism is the conjectured surface-area-minimizing 7-hedral tile.}
\end{figure}
\begin{figure}
	\centering
	\includegraphics[scale=0.6]{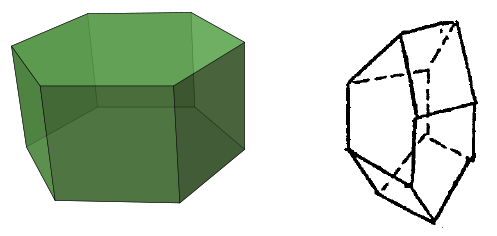}
	\caption{Goldberg's \cite[Fig. 8-VI]{goldbergocta} gabled rhombohedron and the hexagonal prism \cite{wiki} are the conjectured surface-area-minimizing 8-hedral tiles. They have the same surface area.}
	\includegraphics[scale=0.4]{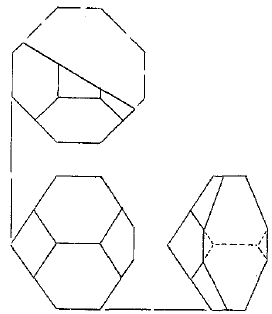}
	\caption{Goldberg's \cite{goldbergdodeca} one half the truncated octahedron is the conjectured surface-area-minimizing 12-hedral tile.}
	\includegraphics[scale=0.4]{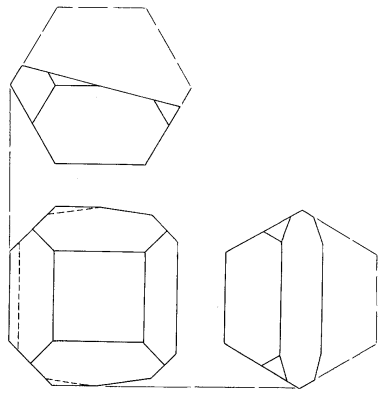}
	\caption{Goldberg's \cite{goldberg>12} Type 13-IV is the conjectured surface-area-minimizing 13-hedral tile. It obtained by cutting a Goldberg's Type 14-IV 14-hedron in half.}
	\includegraphics[scale=0.5]{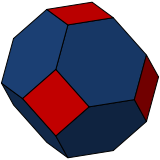}
	\caption{Kelvin's truncated octahedron is the conjectured surface-area-minimizing polyhedral tile. \cite{wiki}}
\end{figure}

Unfortunately, the regular tetrahedron, which is the surface-area-minimizing tetrahedron, does not tile because the dihedral angles of $70.53^\circ$ cannot add up to $360^\circ$ (Fig. \ref{fig:notiletetra}). We provide the best orientation-preserving tetrahedral tile in Theorem \ref{besttetra}, but have not been able to remove the orientation-preserving assumption.

In the known cases $n=5$ and $n=6$, the candidates are surface-area-minimizing unit-volume $n$-hedra and hence, of course, the optimal $n$-hedral tiles. Minkowski \cite{mink} proved that such an $n$-hedron exists, as does Steinitz \cite{Steiz}. (We are not sure whether their arguments imply the existence of a surface-area-minimizing unit-volume $n$-hedral \textit{tile}.)

The case $n=6$ follows immediately from a theorem of Goldberg \cite[p. 230]{goldberg} and also given by Fejes T\'{o}th.

\begin{theorem}
\label{Florianpf} (\cite[pp. 174 - 180]{ftoth}, \emph{see} \cite[pp. 212 - 213]{florian}).
If F denotes the surface area and V the volume of a three-dimensional convex polyhedron with f faces, then 
\[ \frac{F^3}{V^2} \geq 54(f-2)\tan{\omega_f}(4\sin^2{\omega_f-1})  \]
where $\omega_f = \pi f/6(f -2)$. Equality holds only for the regular tetrahedron, the cube, and the regular dodecahedron.
\end{theorem}

Regarding $n=7$, Goldberg \cite{goldberg} claims that the right regular pentagonal prism is the surface-area-minimizing 7-hedron. However, the proof, which was given by Lindel\"{o}f, is $-$ in Lindel\"{o}f's words $-$ "only tentative". Furthermore, regular pentagons cannot tile the plane. Therefore, we cannot tile $\R^3$ with the right regular pentagonal prism. The Cairo and Prismatic pentagons (Fig. \ref{fig:pentile}) have recently been proved by Chung et al. \cite[Thm. 3.5]{pen11} as the best pentagonal planar tiles. They are circumscribed about a circle, with three angles of $2\pi/3$ and two angles of $\pi / 2$, adjacent in the Prismatic pentagon and non-adjacent in the Cairo pentagon. We conjecture a right Cairo or Prismatic prism is the surface-area-minimizing 7-hedral tile.

\begin{figure}
	\centering
	 \includegraphics[scale=0.3]{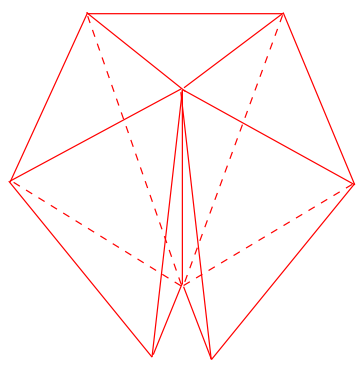}
	\caption{Because the dihedral angles ($70.53^\circ$) of a regular tetrahedron cannot add up to $360^\circ$, the regular tetrahedron does not tile. There is a small gap. \cite{ungor}}
	\label{fig:notiletetra}
\end{figure}

For $n=8$, Goldberg \cite{goldberg} shows that the regular octahedron does not minimize surface area, supporting his conjecture that the surface area minimizer cannot have vertices of degree greater than three. We found that the gabled rhombodecahedron has the same surface area as the regular octahedron (which does not tile) and hexagonal prism (which tiles). Moreover, it has less surface area than the gyrobifastigium suggested by Li et al. \cite[p. 30]{g10}. The gabled rhombodecahedron is distinguished among Goldberg's \cite{goldbergocta} octahedral tiles by having all vertices of degree three.
 
The enneahedron and decahedron are inspired by two of the eight nontrivial geodesic nets on the sphere meeting in threes at $2\pi/3$, classified by Heppes (see \cite{taylornets} and \cite[pp. 132]{morggeo}), although these polyhedra inscribed in spheres are not circumscribed about spheres as surface area minimizers would be. We do not know if any such polyhedra tile space. The conjectured 13-hedron is distinguished by having all vertices of degree three \cite{goldberg>12}.

\begin{figure}
	\centering
	\includegraphics[scale=0.7]{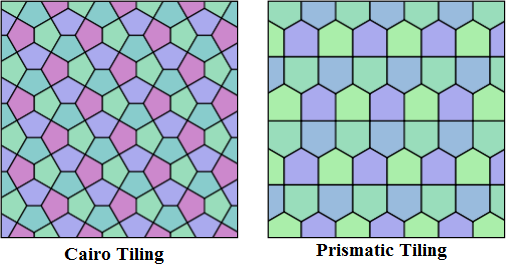}
	\caption{The Cairo and Prismatic pentagons have recently been proved (Chung et al. \cite[Thm. 3.5]{pen11}) as the best pentagonal planar tiles.}
	\label{fig:pentile}
\end{figure}

For any $n \geq 14$, we follow the famous Kelvin Conjecture $-$ that the truncated octahedron is the surface-area-minimizing $n$-hedron that tiles space. Table \ref{tab:poly} gives the surface areas of the conjectured minimizers, computed using Proposition \ref{bestheight} and the Quickhull algorithm \cite{qhull}. Table \ref{tab:comp} shows surface areas of competing 12 and 13-hedra, Rhombic Dodecahedron, Elongated Dodecahedron, Goldberg's Type 13-I, and Type 13-II \cite{goldberg>12}. Note that an $n_0$-hedron may be considered a (degenerate) $n$-hedron for any $n>n_0$ by subdividing its faces, as in \ref{best3Dtiles}(7) and (10).

\begin{table}[ht]
	\centering
    \begin{tabular}{|c|c|}
        \hline
        \begin{tabular}[x]{@{}c@{}} $n=4$ \\ One third triangular prism \end{tabular} & \begin{tabular}[x]{@{}c@{}}\textbf{7.4126}\\ \includegraphics[scale=0.2]{besttetra.png} \end{tabular} \\ \hline
        \begin{tabular}[x]{@{}c@{}} $n=5$ \\ A triangular prism \end{tabular} & \begin{tabular}[x]{@{}c@{}}\textbf{6.5467}\\ \includegraphics[scale=0.2]{righttriangularprism.png} \end{tabular} \\ \hline
        \begin{tabular}[x]{@{}c@{}} $n=6$ \\ Cube \end{tabular} & \begin{tabular}[x]{@{}c@{}}\textbf{6.0000}\\ \includegraphics[scale=0.2]{cube.png} \end{tabular} \\ \hline
        \begin{tabular}[x]{@{}c@{}} $n=7$ \\ Cairo pentagonal prism \end{tabular} & \begin{tabular}[x]{@{}c@{}}\textbf{5.8629}\\ \includegraphics[scale=0.2]{7-hedra.png} \end{tabular} \\ \hline
        \begin{tabular}[x]{@{}c@{}} $n=8$ \\  A Hexagonal prism \end{tabular} & \begin{tabular}[x]{@{}c@{}}\textbf{5.7191}\\ \includegraphics[scale=0.2]{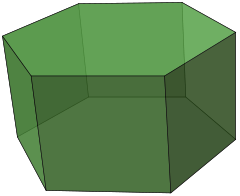} \end{tabular} \\ \hline
        \begin{tabular}[x]{@{}c@{}} $n=9$ \\  An Enneahedron \end{tabular} & \begin{tabular}[x]{@{}c@{}}\textbf{5.5299}\\ \includegraphics[scale=0.2]{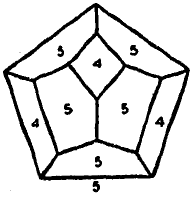} \end{tabular} \\ \hline
        \begin{tabular}[x]{@{}c@{}} $n=10$ and 11 \\ Decahedral barrel \end{tabular} & \begin{tabular}[x]{@{}c@{}}\textbf{5.4434}\\ \includegraphics[scale=0.2]{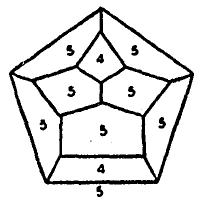} \end{tabular} \\ \hline
        \begin{tabular}[x]{@{}c@{}} $n=12$ \\  Half truncated octahedron \end{tabular} & \begin{tabular}[x]{@{}c@{}}\textbf{5.3199}\\ \includegraphics[scale=0.2]{12-hedra.png} \end{tabular} \\ \hline\\
        \begin{tabular}[x]{@{}c@{}} $n=13$ \\ Goldberg's \cite{goldberg>12} Type 13-IV \end{tabular} & \begin{tabular}[x]{@{}c@{}}\textbf{5.3189}\\ \includegraphics[scale=0.2]{13v.png} \end{tabular} \\ \hline
        \begin{tabular}[x]{@{}c@{}} $n=14$ \\ Kelvin's truncated octahedron \end{tabular} & \begin{tabular}[x]{@{}c@{}}\textbf{5.3147}\\ \includegraphics[scale=0.2]{truncatedoctahedron.jpg} \end{tabular} \\ \hline
    \end{tabular}\\ \vspace{4 mm}
    \caption{Our conjectured surface-area-minimizing unit-volume $n$-hedral tiles.}
    \label{tab:poly}
\end{table} 

\begin{table}[ht]
	\centering
    \begin{tabular}{|c|c|c|c|}
        \hline
        \begin{tabular}[x]{@{}c@{}} $n=12$ \\ Rhombic Dodecahedron \end{tabular} & \begin{tabular}[x]{@{}c@{}}\textbf{5.3454}\\ \includegraphics[scale=0.2]{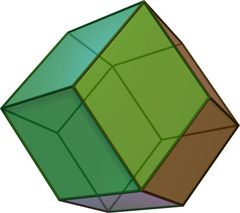} \end{tabular} & \begin{tabular}[x]{@{}c@{}} $n=12$ \\ Elongated Dodecahedron \end{tabular} & \begin{tabular}[x]{@{}c@{}}\textbf{5.4932}\\ \includegraphics[scale=0.2]{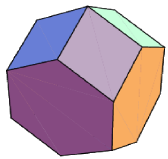} \end{tabular}\\ \hline
        \begin{tabular}[x]{@{}c@{}} $n=13$ \\ Goldberg's Type 13-I \end{tabular} & \begin{tabular}[x]{@{}c@{}}\textbf{5.3640}\\ \includegraphics[scale=0.2]{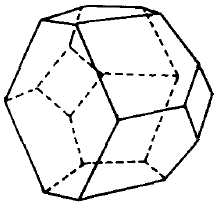} \end{tabular} & \begin{tabular}[x]{@{}c@{}} $n=13$ \\ Goldberg's Type 13-II \end{tabular} & \begin{tabular}[x]{@{}c@{}}\textbf{6.8813}\\ \includegraphics[scale=0.2]{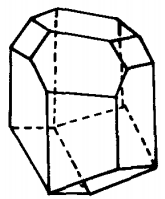} \end{tabular}\\ \hline
        \end{tabular}\\ \vspace{4 mm}
    \caption{Table showing the surface area of competing 12 and 13-hedral tiles.}
    \label{tab:comp}
\end{table}

On the other hand, the following proposition shows that you can always reduce surface area by a small truncation and rescaling, but the resulting polyhedron may not tile. We think the truncated octahedron as far as you can go and still tile.

\begin{proposition}
\label{trunc}
A slight truncation at any strictly convex vertex and rescaling to the original volume reduces the surface area of a polyhedron.
\end{proposition}

\begin{proof}
Instead of rescaling, we show the decrease of the scale invariant area-volume ratio $A^3 / V^2$. Under truncation by a distance $t$, the logarithmic derivative 
$$
\frac{3A'}{A} - \frac{2V'}{V}
$$
is negative for all sufficiently small $t$ because $A'$ is proportional to $-t$, while $V'$ is proportional to $-t^2$.
\end{proof}

Heppes drew our attention to Wolfram Online's \cite{wolfram} discussion of polyhedral tiles. It notes the extensive categorization of polyhedral tiles by Goldberg [G1-G7]. Gr\"{u}nbaum and Shephard \cite{grunbaum} and Wells \cite{wells} discuss the known polyhedral tiles pre-1980, when the maximal $n$ for $n$-hedral tiles was believed to be 26. In 1980, P. Engel \cite[pp. 234-235]{wells} found 172 additional polyhedral tiles with 17 to 38 faces, and more polyhedral tiles have been found subsequently.
\section{Existence of a surface-area-minimizing tile}
\label{existence}
For fixed $n$, Minkowski \cite{mink} proved that among convex polyhedra, there exists a surface-area-minimizing $n$-hedron. We show that if we assume the polyhedron tiles space, then there exists a surface-area-minimizing convex polyhedral tile.

\begin{definition}
\emph{A polyhedron is} nondegenerate \emph{if it does not have any unnecessary edges.}

\emph{The furthest distance between two vertices is the} diameter \emph{of a polyhedron.}

\indent \emph{We call two polyhedra $P$ and $Q$ combinatorially equivalent if there exists a bijection $f$ between the set of the vertices of $P$ and $Q$ such that:}
\begin{enumerate}
\item $v_1v_2$ is an edge of $P$ if and only if $f(v_1)f(v_2)$ is an edge $Q$.
\item $v_1, \dotsc, v_k$ is a face of $P$ if and only if $f(v_1), \dotsc, f(v_k)$ is a face $Q$.
\end{enumerate}
\end{definition}

\begin{proposition}
\label{types}
For any $n$, there are a finite number of combinatorial types of $n$-hedra. 
\end{proposition}

\begin{proof}

Fix $n$. First, a $n$-hedron's face can have at most $(n-1)$-edges. Assume, on the contrary, that a $n$-hedron contains an $n$-gon. Then since each edge is shared by two faces and two faces share at most one edge, there are at least $n+1$ faces in the $n$-hedron, which is a contradiction. This means that the biggest face can have $n-1$ edges and the smallest is a triangle (3 edges). 

Therefore, we have $n-3$ choices for each face. Hence, the number of possible combinations of $n-3$ faces is equivalent to the number of solutions to the equation 
$$
x_3+x_4+...+x_{n-1} = n
$$
where $x_i$ corresponds to the number of faces with $i$ edges. The number of solutions to the equation is ${2n-4 \choose n}$. It follows that for each combination, we can arrange the faces in a finite number of ways. Therefore, there are a finite number of combinatorial types.
\end{proof}

\begin{remark}
\emph{Not all possible combinations of faces can make a polyhedron. For example for $n=5$, it is possible to have 6 combinations of different faces, but in Proposition \ref{fivefaceopt}, we will prove that the only combinatorial types are either triangular prisms or quadrilateral pyramids.}
\end{remark}

\begin{theorem}
\label{existspoly}
For a fixed $n$, there exists a surface-area-minimizing unit-volume convex n-hedral tile.
\end{theorem}

The minimizer could be a degenerate $n$-hedron (with fewer than $n$ faces), as we conjecture occurs for $n>14$ (Conj. \ref{best3Dtiles} (10)).

\begin{proof}
Take a sequence of unit-volume convex $n$-hedral tiles with areas approaching the infimum. We may assume that the areas are bounded by $P_0$. By standard compactness results, it suffices to show that the diameters are bounded.

Consider a unit-volume convex polyhedron. Take the slice of largest area $a_0$ perpendicular to the diameter $D$. Consider the pyramid with based $a_0$ and the apex at the most distant end of the diameter. By convexity, the pyramid lies insider the polyhedron. Therefore,
$$
1 \geq \left(\frac{1}{3}\right)a_0 \frac{D}{2}
$$
and
$$
a_0 \leq \frac{6}{D}.
$$
For every slice perpendicular to the diameter, by the isoperimetric inequality, the perimeter $p$ and area $a$ satisfy
$$
p \geq \sqrt{4 \pi a}.
$$
Since $\sqrt{a} \geq a / \sqrt{a_0}$, we have
$$
\sqrt{4 \pi a} \geq \frac{a\sqrt{4 \pi}}{\sqrt{6/D}} = a\sqrt{\frac{2 \pi D}{3}}.
$$
Integrating over all slices, the area becomes the volume which equals 1 and the perimeter-area $P_0$ satisfies
$$
P_0 \geq \sqrt{\frac{2 \pi D}{3}}.
$$
Therefore,
$$
D \leq \frac{3P^2}{2\pi},
$$
as desired.
\end{proof}

\begin{remark}
\emph{In general an area-minimizing $n$-hedral tile need not be unique. Indeed, for $n = 8$, the conjectured gabled rhombohedron and hexagonal prism have the same surface area.}
\end{remark}

\section{Properties of Prisms}
\label{secprism}

In this section, we give some properties of prisms, which are useful in the next section. We begin by giving a definition of prisms. Then we characterize prisms by showing that if a polyhedron has two $n$-gonal bases and $n$ quadrilateral faces, then it must be a prism (Prop. \ref{combinatorial_prism_face_3} and \ref{combinatorial_prism_face_n}). Moreover, we show that a prism with a regular polygonal base uniquely minimizes surface area among all prisms of fixed volume and number of faces and give a way to calculate the surface area and optimal height (Prop. \ref{bestheight}). Lastly, in Proposition \ref{montile}, we relate tiling of the plane with tiling of space in order to prove that a certain hexagonal prism is the surface-area-minimizing prism (Prop. \ref{hexbest}).

\begin{definition}
\emph{A} prism \emph{is a polyhedron consisting of a polygonal planar base, a translation of that base to another plane, and edges between corresponding vertices.}
\end{definition}

\begin{remark}
\emph{Bernd Sturmfels \cite{sturmfels} asked us the following question: given a specific combinatorial type for some $n$-hedron, can we determine whether there exists a tile of that type. We conjecture that the pentagonal pyramid is the combinatorial polyhedron with the fewest faces which does not tile. Wolfram Online \cite{wolfram} remarks that there are no known pentagonal pyramids which tile.}
\end{remark}

The next two propositions characterize when we know that a $n$-hedron must be a combinatorial prism.

\begin{proposition}
\label{combinatorial_prism_face_3}
Let $P$ be a nondegenerate polyhedron with three quadrilateral faces and two triangular faces.  Then $P$ is a combinatorial triangular prism. 
\end{proposition}

\begin{proof}
Since each edge lies on two faces, the total number of edges is 9. By Euler's formula, the number of vertices is 6. Since the sum over the faces of the number of vertices is 18, each vertex must have degree 3. (By the nondegeneracy hypothesis, no vertex can have degree 2.)

Suppose that the triangular faces $\bigtriangleup ABC$ and $\bigtriangleup ABY$ meet. Because each vertex has degree 3, they must share an edge, as in Figure \ref{fig:tprism1}. The other faces at edges $AC$ and $BC$ must be quadrilaterals. Quadrilateral $ACXY$ has vertices $X$ and $Y$, distinct because the prism has degree 3. It follows that the vertex $B$ is not of degree 3, a contradiction. Therefore, the triangular faces are disjoint and the polyhedron is a combinatorial triangular prism, as desired. 
\end{proof}

Proposition \ref{combinatorial_prism_face_n} shows that more generally a nondegenerate polyhedron with $n$ quadrilateral faces and two $n$-gonal faces is a combinatorial $n$-gonal prism. The proof is similar to the proof of Proposition \ref{combinatorial_prism_face_3}.

\begin{proposition}
\label{combinatorial_prism_face_n}
Let $P$ be a nondegenerate polyhedron with $n$ quadrilateral faces and two $n$-gonal faces.  Then $P$ is a combinatorial $n$-gonal prism. 
\end{proposition}

\begin{proof}
By the same argument in Proposition \ref{combinatorial_prism_face_3}, we can show that every vertex has degree 3 and that $V=2n$ and $E=3n$.
\newline

\noindent \emph{(Case 1)}: $n=4$.
\newline

Since no vertex can have degree greater than three, it must be the case that two of the faces do not share a vertex. Since the six faces of this polyhedra will be quadrilaterals, we can identify any two faces as bases. 
\newline

\noindent \emph{(Case 2)}: $n \geq 5$
\newline

Suppose that the two $n$-gonal faces meet. If they only share one vertex, then the degree of this vertex is at least four, a contradiction. So they should meet at an edge. Let us call this edge $cd$ and the two $n$-gonal faces $a_1a_2 \dotsc a_{n-2}cd$ and $b_1b_2 \dotsc b_{n-2}cd$. $c$ is contained in the edges $ca_{n-2}$, $cb_{n-2}$, and $cd$. Therefore, there exists a quadrilateral face containing the edges $ca_{n-2}$ and $cb_{n-2}$, namely $ca_{n-2}xb_{n-2}$. Similarly, there exists a vertex $y$ such that $db_1ya_1$ is a face of $P$. If $x=y$, then the degree of $x$ is at least four, a contradiction. So $x$ and $y$ are distinct.

Now note that since $b_1$ is contained in the three edges $b_1d$, $b_1y$, and $b_1b_2$, there exists a face containing the edges $b_1b_2$ and $b_1y$.  This face must be a quadrilateral, so there exists a vertex $z$ such that $b_2b_1yz$ is a face of $P$.  Since there are $2n$ vertices of $P$, $z \in \{a_1, \dotsc ,a_{n-2},b_1, \dotsc ,c_{n-2},c,d,x,y\}$. Moreover, since two faces meet at most at two vertices, $z \in \{b_3, \dotsc ,b_{n-2},x\}$. It follows that $\deg{z}$ is at least four, a contradiction.  Therefore, the two $n$-gonal faces do not share an edge, and it follows that they cannot meet.

We now show that $P$ is a combinatorial $n$-gonal prism. Let $a_1a_2 \dotsc a_n$ be an $n$-gonal face described above. Let the other $n$-gonal face have vertices $b_1,b_2,\dotsc ,b_n$. By permuting the vertices $b_1,b_2,\dotsc, b_n$, we may assume that $a_ib_i$ is an edge of $P$ for each $i=1,2, \dotsc ,n$.  $a_ia_{i+1}$ is contained in a face of $P$ other than $a_1a_2 \dotsc a_n$. Since this face will contain the edges $a_ib_i$ and $a_{i+1}b_{i+1}$, we conclude that $a_ib_ib_{i+1}a_{i+1}$ is a face of $P$. Therefore, $b_ib_{i+1}$ is an edge of $P$.  Hence, $b_1b_2\dotsc b_n$ is a face of $P$. From this map, it is clear $P$ is a combinatorial $n$-gonal prism, as desired.
\end{proof}

\begin{figure}
	\centering
	\includegraphics[scale=0.6]{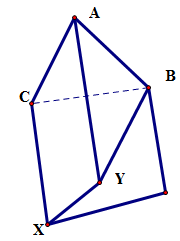}
	\caption{Two triangular faces cannot meet in a nondegenerate polyhedron with three quadrilateral faces and two triangular faces.}
	\label{fig:tprism1}
\end{figure}

The following proposition gives the optimal height for any right regular prism:

\begin{proposition}
\label{bestheight}
The optimal unit-volume prism with a base similar to a region $R$ of area $A_0$ and perimeter $P_0$ is a right prism of height $h=(4\sqrt{A_0}/P_0)^{2/3}$ and surface area $S = 3({P_0^2}/{2A_0})^{1/3}$. If the base is a regular polygon, it uniquely minimizes surface area among all prisms of fixed volume and number of faces.
\end{proposition}

\begin{proof}
Since the top is a translation of the bottom, we may assume that both are horizontal. Since shearing a right prism preserves volume but increases surface area, we may assume that our prism is a right prism. A simple calculus computation shows that the optimal right prism has height and surface area as asserted. Since a regular $n$-gon uniquely minimizes perimeter for given area, the right $n$-gonal prism of optimal dimensions uniquely minimizes surface area among all prism of fixed volume and number of faces.
\end{proof}

The next proposition gives an example of how we can relate tiling of the plane with tiling of space. We use Proposition \ref{montile} and Hales' honeycomb theorem \cite[Thm. 1-A]{hales} to prove that the hexagonal prism is the surface-area-minimizing prism.

\begin{proposition}
\label{montile}
Given $n \geq 5$, a monohedral tiling of space by a unit-volume right prisms with $n$ faces is surface-area-minimizing among prisms if and only if the bases are perimeter-minimizing tilings of parallel planes by fixed-area $(n-2)$-gons and the height is optimal as in Proposition \ref{bestheight}.
\end{proposition}

\begin{proof}
We claim that bases must match up with bases and sides with sides. For $n \neq 6$, this is trivial. For $n = 6$, the prism is a cube and the claim is even more trivial. Therefore, the bases tile parallel planes. Furthermore, the bases minimize perimeter for fixed area if and only if the prisms minimize perimeter for fixed volume.
\end{proof}

\begin{remark}
\emph{Proposition \ref{montile} assures that the surface-area-minimizing tile which is combinatorial prism of seven faces is the Cairo prism.}
\end{remark}

\begin{proposition}
\label{hexbest}
A right regular hexagonal prism of base length $(2/9)^{1/3}$ and height $2^{1/3}3^{-1/6}$ provides the least-surface area tiling of space by unit-volume prisms. Its surface area is $2^{2/3}3^{7/6}$.
\end{proposition}

\begin{proof}
Hales' honeycomb theorem \cite[Thm. 1-A]{hales} says that a regular hexagon provides the least-perimeter way to tile the plane into equal parts. By Proposition \ref{montile}, a regular hexagonal prism is the least-surface-area way to tile space by equal volume prisms. The best right regular hexagonal prism has height given by Proposition \ref{bestheight}. Since the base length of a unit-volume right regular hexagonal prism is determined by its height, we have the desired result. 
\end{proof}

\section{The surface-area-minimizing tetrahedron and 5-hedron tiles}
\label{5and4hedra}

The regular tetrahedron is the surface-area-minimizing tetrahedron by Theorem \ref{Florianpf}, but, unfortunately, does not tile space (Fig. \ref{fig:notiletetra}). While the problem of tetrahedral tilings has been considered in the literature, there does not seem to be a discussion of \textit{surface-area-minimizing} tetrahedral tiles. In this section, we use Sommerville's classification of space-filing tetrahedra to find the surface-area-minimizing tetrahedron. However, we are unable to remove the orientation-preserving assumption. We first define an orientation-preserving tiling as follows:

\begin{definition}
\label{propertiling}
\emph{A tiling is} orientation preserving \emph{if any two tiles are equivalent under an orientation-preserving isometry of $\R^3.$}
\end{definition}

Sommerville \cite[p.57]{somville} describes four types of tetrahedral tiles and claims that, "in addition to these four, no tetrahedral tiles exist in euclidean space". Edmonds \cite{edmonds} addresses some concerns about Sommerville's proof and proves that Sommerville's four candidates are indeed the only four face-to-face, orientation-preserving tiles. The No. 1 tetrahedron is given by cutting a triangular prism into three (See Fig. \ref{fig:tetraprism}). The No. 2 tetrahedron is given by cutting No. 1 or cutting No. 3 in half (Fig. \ref{fig:tetra2}). The No. 3 tetrahedron is given by cutting a square pyramid in half across the diagonal of the base (Fig. \ref{fig:tetra3}). This means No. 3 is 1/12 a cube. Note that No. 3 was incorrectly suggested  by Li et al. \cite{g10} as a surface-area-minimizing tetrahedral tile. Lastly, the No. 4 tetrahedron is given by cutting No. 1 into 4 (Fig. \ref{fig:tetra4}).

\begin{figure}
	\centering
    \includegraphics[scale=0.7]{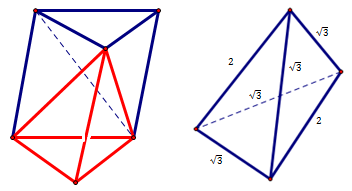}
	\caption{The tetrahedron (Sommerville No. 1) formed by four isosceles right triangles with two sides of $\sqrt{3}$ and one side of 2 minimizes surface area among all orientation-preserving tetrahedral tiles \cite[Fig. 7]{somville}.}
	\label{fig:tetraprism}
	\includegraphics[scale=0.7]{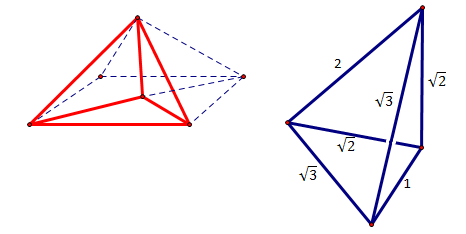}
	\caption{No. 2 tetrahedron is given by cutting No. 3 in half. \cite[Fig. 8]{somville}.}
	\label{fig:tetra2}
	\includegraphics[scale=0.7]{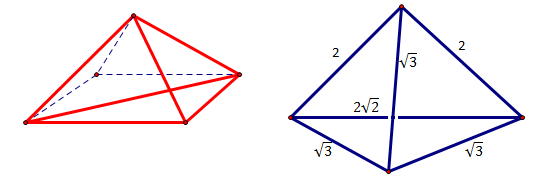}
	\caption{No. 3 tetrahedron is given by cutting a square pyramid into two. \cite[Fig. 9]{somville}.}
	\label{fig:tetra3}
	\includegraphics[scale=0.7]{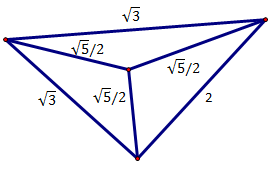}
	\caption{No. 4 tetrahedron is given by cutting No. 1 into 4. \cite[Fig. 10]{somville}.}
	\label{fig:tetra4}
\end{figure}

Goldberg \cite{goldbergtetra} considered more general tetrahedral tilings (which are not face-to-face) and found infinitely many families of them. Edmonds does not consider tilings which are not orientation-preserving. Further investigation is needed regarding what is known about nonorientation-preserving tilings, and whether the orientation-preserving hypothesis can be removed from the Theorem \ref{besttetra}. 

Marjorie Senechal \cite{senechal} provides an excellent survey on tetrahedral tiles. Senechal explains that Sommerville's initial consideration of this question goes back to an error made by a student. The student stated that three tetrahedra which divide a triangular prism are congruent, though he meant equal volume. This prompted Sommerville's initial study of congruent tetrahedra which tile space. Senechal points out that Sommerville seems to consider only orientation-preserving, face-to-face tetrahedral tilings, and she stresses the need for more consideration of the problem.

We now proceed to show that the No. 1 tetrahedron provides the optimal orientation-preserving tetrahedral tiling of space.

\begin{theorem}
\label{besttetra}
Let $T$ be the No. 1 tetrahedron formed by four isosceles right triangles with two sides of $\sqrt{3}$ and one side of 2 (Fig. \ref{fig:tetraprism}). Then $T$ provides the least-surface-area unit-volume orientation-preserving tetrahedral tiling.
\end{theorem}

\begin{proof}
Since Sommerville provides edge lengths and dihedral angles for each of the four types, we scaled the various tetrahedra to unit volume and calculated the surface area of each. The four types had surface areas of $7.4126, 7.9635, 8.1802,$ and $10.3646$ (to four decimal places), respectively. Thus, $T$ is the surface-area-minimizing orientation-preserving tetrahedral tile.
\end{proof}

\begin{remark}
\label{sumofdihedral}
\emph{For all prisms, the sum of all dihedral angles is a multiple of 360. This does not hold for every polyhedron that tiles $\R^3$, as shown by Sommerville's tetrahedra (as seen in \ref{besttetra}).}
\end{remark}

Although Conjecture \ref{best3Dtiles}(2) for $n=5$ is well known, there seems to be no nice proof in the literature. The more specific problem of tiling space with prisms was put forth by Steiner (\cite{Steiner2}; see \cite[p. 209]{florian}) who conjectured that a right prism with a regular polygonal base was surface area minimizing among all combinatorial prisms. Steinitz apparently proved the conjecture for triangular prisms but the result was never published (see \cite[p. 209]{florian}). Brass, Moser, and Pach \cite{disgeo} assert that the optimal $n$-hedron is known for $n \leq 7$ but do not provide candidates, though they do reference Goldberg \cite{goldberg}. Goldberg says that the optimal candidate among 5-hedra is known, but offers no proof or specific reference in his paper. We are happy to add our proof and Corollary \ref{triprismtile} to the literature.

Earlier, Sucksdorff \cite{french} gave a proof which Florian \cite[p. 211]{florian} calls "very troublesome". Sucksdorff first eliminates other combinatorial types by noting that the well-known best representative, a square pyramid, has more surface area than the optimal triangular prism. Then follow eighteen pages of algebraic and trigonometric inequalities to show that the right equilateral triangular prism of optimal height minimizes surface area in its combinatorial type. The editor, M. Catalan, appends a note that Sucksdorff's conclusion agrees with the theorem published by Lindel\"{o}f \cite{lind} twelve years later, of which Sucksdorff was apparently unaware. The editor had heard of the result somewhere, from "Mr. Steiner, I believe." We thank Bill Dunbar for help reading the original French.

Our proof of the least surface area 5-hedron begins by first showing that the faces characterize a combinatorial triangular prism (Prop. \ref{combinatorial_prism_face_3}). Then we show that a polyhedron with five faces is combinatorially equivalent to a square pyramid or a triangular prism (Prop. \ref{fivefaceopt}). Furthermore, we prove that the square pyramid is the least-surface-area combinatorial pyramid (Prop. \ref{square-pyramid}) and find a triangular prism that has less surface area than the square pyramid (Prop. \ref{optprism}). Therefore, the best 5-hedron must be a combinatorial triangular prism. By computation, we eliminated non-convex 5-hedra. Therefore, the most efficient must be convex. Finally, using Lindel\"{o}f's Theorem (Thm. \ref{linde}), we show that the 5-hedron with the least surface area is the right equilateral triangular prism (Thm. \ref{bestfivepoly}).

In section \ref{secprism}, we gave the following proposition, which shows that faces characterize a combinatorial triangular prism.
\newline
\newline
\textbf{Proposition \ref{combinatorial_prism_face_3}.} \emph{Let P be a nondegenerate polyhedron with three quadrilateral faces and two triangular faces.  Then $P$ is a combinatorial triangular prism.}
\newline

We now show a nondegenerate polyhedron with five faces is combinatorially equivalent to a square pyramid or a triangular prism by using Euler's formula to limit the number of possible combinations of quadrilateral and triangular faces to three. Then we show one case is impossible and apply Proposition \ref{combinatorial_prism_face_3} to complete the proof.

\begin{proposition}
\label{fivefaceopt}
A nondegenerate polyhedron with five faces is combinatorially equivalent to a square pyramid or a triangular prism.
\end{proposition}

\begin{proof}
Because $P$ has five faces and is nondegenerate, each face is either a triangle or a quadrilateral. Let $a$ be the number of triangular faces and $b$ be the number of quadrilateral faces. Since $P$ has five faces, we have $a+b=5$. Let $V$ be the number of vertices of $P$ and $E$ be the number of edges of $P$. By Euler's formula, we have $V-E+5=2$. By calculating the sum of the number of edges of each face of $P$, we have $2E=3a+4b$. Therefore, $a$ is even.
\newline

\noindent\textit{(Case 1):} $a=0$ and $b=5$.
\newline

From the above formulas, we have $V=7$ and $E=10$. By counting the number of edges from each vertex, we have that the sum of degrees of vertices of $P$ is $2E=20$. By the Pigeonhole principle, there exists a vertex which has degree less than or equal to $20/7$. Since every degree is at least three, we get a contradiction.
\newline

\noindent\textit{(Case 2):} $a=2$ and $b=3$.
\newline

By Proposition \ref{combinatorial_prism_face_3}, $P$ is a combinatorial triangular prism.
\newline

\noindent\textit{(Case 3):} $a=4$ and $b=1$.
\newline

From the above formulas, we have $V=5$ and it easily follows that $P$ is a quadrilateral pyramid. Therefore, we have shown that $P$ is either a combinatorial triangular prism or quadrilateral pyramid.
\end{proof}

Next, we give a lower bound on the surface area of a given pyramid and use it to show that the quadrilateral pyramid with a square base has the least surface area of among quadrilateral pyramids. 

\begin{lemma}
\label{side-surface}
Let $P$ be a pyramid with apex $V$, base $A_1A_2...A_n$ and height $h$. Suppose that the base has area $S$ and perimeter $p$, then the sum of the areas of side faces of $P$ is greater than or equal to $(1/2)\sqrt{(2S)^2+p^2h^2}$. Equality holds if and only if the base is circumscribed about a circle and the foot of the perpendicular line from $V$ to the base is the center of the circumscribing sphere.
\end{lemma}

\begin{proof}
Let $B$ be the foot of the perpendicular line from $V$ to the base.  Let $a_1,a_2,...,a_n$ be the lengths of the sides of the base.  Let $x_1,x_2,...,x_n$ be the distances from $B$ to the sides of the base.  Then we have $\sum_i \pm a_ix_i=2S$. This implies that $\sum_i a_ix_i\geq2S$. Equality holds when $B$ lies in the interior of the base.
The sum of areas of side faces of $P$ is given by
$$
\frac{1}{2}\sum_i a_i\sqrt{x_i^2+h^2}=\frac{1}{2}\sum_i \sqrt{\left(a_ix_i\right)^2+ \left(a_ih\right)^2}.
$$
By the triangle inequality,
$$
\sum_i \sqrt{\left(a_ix_i \right)^2+ \left(a_ih \right)^2}\geq \sqrt{\left(\sum_i a_ix_i\right)^2+ \left(\sum_i a_ih\right)^2}.
$$
Together with the inequality $\sum_i a_ix_i\geq2S$, we get the desired inequality. It is easy to verify the equality condition.
\end{proof}

\begin{proposition}
\label{square-pyramid}
Let $P$ be a unit-volume quadrilateral pyramid.  Then the surface area of $P$ is greater than or equal to $2^{5/3}3^{2/3}$. Equality holds if and only if it is a right regular pyramid with base-length $2^{-1/3}3^{2/3}$ and height $2^{2/3}3^{-1/3}$.
\end{proposition}

\begin{proof}
Let $S$ be the area and $p$ be the perimeter of the base of $P$.  Let $h$ be the height of $P$. Since $P$ has unit volume, we have $Sh=3$. Moreover, for given perimeter, the square is the area maximizer among quadrilaterals. Therefore, $p\geq 4\sqrt{S}$.
From Lemma \ref{side-surface}, the surface area of $P$ is greater than or equal to 
$$
S+\frac{1}{2}\sqrt{(2S)^2+p^2h^2}=S+\frac{1}{2}\sqrt{(2S)^2+\frac{9p^2}{S^2}}.
$$
Furthermore, we have the following inequalities:
$$
S+\frac{1}{2}\sqrt{(2S)^2+\frac{9p^2}{S^2}}\geq S+\frac{1}{2}\sqrt{(2S)^2+\frac{9(16S)}{S^2}}
=S+\sqrt{S^2+\frac{36}{S}}.
$$
Therefore, it suffices to show that 
$$
S+\sqrt{S^2+\frac{36}{S}}\geq 2^{5/3}3^{2/3}
$$
or equivalently that
$$
S^2+\frac{36}{S} \geq \left(2^{5/3}3^{2/3}-S \right)^2.
$$
By direct calculation, this is equivalent to 
$2^{8/3}3^{2/3}S+36/S \geq 2^{10/3}3^{4/3}$.
This follows directly from AM-GM inequality.
It is easy to check the equality condition from the equality condition of AM-GM inequality and Lemma \ref{side-surface}. 
\end{proof}

Proposition \ref{optprism} shows that a triangular prism has less surface area than the square pyramid. Therefore, it has less surface area than any unit-volume quadrilateral pyramid. It follows that the optimal 5-hedral tile must be a combinatorial triangular prism.

\begin{proposition}
\label{optprism}
Let $P$ be the unit-volume right equilateral-triangular prism circumscribed about a sphere and $Q$ be a unit-volume quadrilateral pyramid. Then $P$ has less surface area than $Q$.
\end{proposition}

\begin{proof}
By direct computation, we have that $P$ has base-length $4^{1/3}$ and height $4^{1/3}3^{-1/2}$. $P$ has surface area $2^{1/3}3^{3/2}$. Therefore, by Proposition \ref{square-pyramid}, the triangular prism has less surface area than any unit-volume quadrilateral pyramid.
\end{proof}

Before, we proceed to the main theorem, we use a linear algebra argument to show that the edges of the sides of a triangular prism are either parallel or concur at a point. We then use this lemma in a our main theorem.

\begin{lemma}
\label{combinatorial_triangular_prism_classification}
Let $ABC-DEF$ be a combinatorial triangular prism such that $ABC$ and $DEF$ are triangular faces. Then the lines $AD$, $BE$, and $CF$ are either parallel to each other or concur at a point (Fig. \ref{fig:prismlines}).
\end{lemma}

\begin{proof}
Imagine the prism $ABC-DEF$ is placed in an Euclidean space such that $ABC$ lies in the plane $z=0$. Pick vectors $v_1$, $v_2$ and $v_3$ such that they are parallel to $\overrightarrow{AD}$, $\overrightarrow{BE}$ and $\overrightarrow{CF}$, respectively and they all have $z$ coordinate 1. Consider the vector space $V$ spanned by the vectors $v_1$, $v_2$ and $v_3$.
\newline

\noindent\textit{(Case 1):} $\dim(V)=1$.
\newline

$v_1,v_2$ and $v_3$ are the same.  Therefore $AD$, $BE$ and $CF$ are parallel to each other, as desired.
\newline

\noindent\textit{(Case 2):} $\dim(V)=2$.
\newline

Since the vectors $v_1$, $v_2$ and $v_3$ are not all the same, there exists a vector among them that is different from the others. Without loss of generality, suppose $v_3$ is different from $v_1$ and $v_2$. Then, $v_3$ and $v_1$ span the plane $ACFD$. Hence, $V$ contains the vector $\overrightarrow{AC}$. Similarly, we can show that the vector $\overrightarrow{BC}$ is contained in $V$. Because $\overrightarrow{AC}$, $\overrightarrow{BC}$, and $v_3$ are linearly independent, $\dim(V)=3$, a contradiction.
\newline

\noindent\textit{(Case 3):} $\dim(V)=3$.
\newline

It follows that $v_1,v_2$ and $v_3$ are distinct. Since $v_2$ and $v_3$ span the plane $BCFE$, there exists a real number $\alpha_1$ such that the vector $\overrightarrow{BC}=\alpha_1(v_2-v_3)$.  Similarly, there exist real numbers $\alpha_2$ and $\alpha_3$ such that the vector $\overrightarrow{CA}=\alpha_2(v_3-v_1)$ and the vector $\overrightarrow{AB}=\alpha_3(v_1-v_2)$. Take the sum of these equations. We have
$$
(\alpha_3-\alpha_2)v_1+(\alpha_1-\alpha_2)v_2+(\alpha_2-\alpha_3)v_3=0.
$$
Since $v_1,v_2$ and $v_3$ are linearly independent, $\alpha_1=\alpha_2=\alpha_3(:=\alpha)$. It follows that 
$$
A+\alpha v_1=B+\alpha v_2=C+\alpha v_3.
$$
Therefore, the lines $AD$, $BE$ and $CF$ meet at a point.
\end{proof}

\begin{figure}
	\centering
	\includegraphics[scale=0.7]{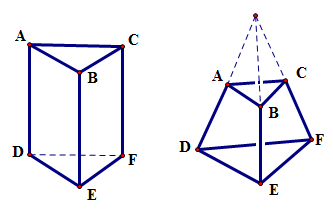}
	\caption{In a combinatorial triangular prism, the lines $AD$, $BE$, and $CF$ are either parallel to each other or concur at a point.}
	\label{fig:prismlines}
\end{figure}

Lorenz Lindel\"{o}f \cite{lind} proved that a surface-area-minimizing $n$-hedron is circumscribed about a sphere, with each face tangent at its centroid.  See the beautiful survey by Florian \cite[pp. 174-180]{florian} and \cite[Prop. 3.1]{pen11} from before we knew about Lindel\"{o}f. For a given combinatorial type, in order to find the surface-area-minimizing polyhedron of that type, it is usually enough to make sure it satisfies Lindel\"{o}f's condition. We prove that the right equilateral-triangular prism minimizes surface area among unit-volume 5-hedra, by showing that if the 5-hedra must satisfy Lindel\"{o}f's conditions, then the only possibility is that it is a right equilateral-triangular prism.

\begin{theorem}[Lindel\"{o}f Theorem \cite{lind}.]
\label{linde}
A necessary condition for a polyhedron P to be the surface-area-minimizing polyhedron is that P circumscribes a sphere, and the inscribed sphere is tangent to all the faces of P at their respective centroids.
\end{theorem} 
 
\begin{theorem}
\label{bestfivepoly}
The right equilateral-triangular prism circumscribed about a sphere minimizes surface area among unit-volume 5-hedra.
\end{theorem}

\begin{proof}
A surface-area-minimizing 5-hedron $X$ exists \cite{mink}. By Proposition \ref{combinatorial_prism_face_3}, we may assume that it is nondegenerate. By Lindel\"{o}f's Theorem \cite{lind}, $X$ is circumscribed about a sphere tangent to each face of $X$ at its centroid. By Proposition \ref{optprism}, $X$ cannot be a square pyramid; therefore by Proposition \ref{fivefaceopt}, $X$ is a combinatorial triangular prism.  Define $ABC$ and $DEF$ as the triangular bases of $X$ and $AD$, $BE$, and $CF$ as the edges. To simplify notation, we refer to the bases $ABC$ and $DEF$ as $B_1$ and $B_2$, respectively and the three quadrilateral faces - $ABED$, $BCFE$, and $CADF$ - as $Q_3$, $Q_4$ and $Q_5$, respectively (Fig. \ref{fig:tprism}).

\begin{figure}
	\centering
	\includegraphics[scale=0.5]{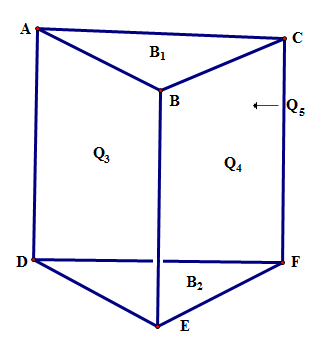}\\
	\caption{By Proposition \ref{optprism}, the surface-area-minimizing 5-hedron $X$ cannot be a square pyramid; therefore by Proposition \ref{fivefaceopt}, $X$ is a combinatorial triangular prism.}
	\label{fig:tprism}
    \includegraphics[scale=0.7]{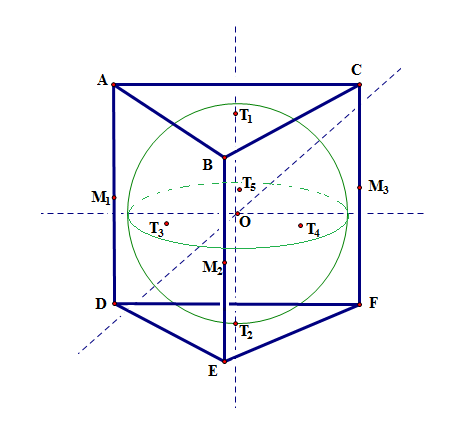}
	\caption{The right equilateral-triangular prism circumscribed about a sphere tangent to each face at its centroid minimizes surface area among unit-volume 5-hedra.}
	\label{fig:besttprism}
\end{figure}

Let $O$ be the center of a sphere inscribed in $X$. Let $T_1, T_2, T_3, T_4$ and $T_5$ be the touching points between the sphere and faces $B_1, B_2, Q_3, Q_4$ and $Q_5$, respectively. Finally, let $M_1$, $M_2$ and $M_3$ be midpoints of $AD$, $BE$ and $CF$, respectively. Imagine we place $X$ in Euclidean space such that $O$ is at the origin (Fig. \ref{fig:besttprism}). 
\newline
\newline
\noindent \textbf{(Step 1)} The midpoint of $T_1T_2$ is the centroid of $T_3T_4T_5$.
\newline
\newline
This follows from the observation that both of them are the centroid of $X$.
\newline
\newline
\noindent \textbf{(Step 2)} The quadrilaterals $M_1T_3T_4T_5$, $M_2T_3T_5T_4$, and $M_3T_4T_3T_5$ are parallelograms.
\newline
\newline
Since $T_3$ is the centroid of $Q_3$, we have that $M_1+M_2=2T_3$. Similarly, we have $M_2+M_3=2T_4$ and $M_3+M_1=2T_5$. By solving this linear equation for $M_1$, $M_2$, and $M_3$, we have $M_1=T_5+T_3-T_4$, $M_2=T_3+T_4-T_5$, and $M_3=T_4+T_5-T_3$, as desired.
\newline
\newline
\noindent \textbf{(Step 3)} $T_3T_4T_5$ is an equilateral triangle.
\newline
\newline
Observe that the face $BEFC$ is perpendicular to the line $OT_4$. Therefore, $\overrightarrow{OT_4} \cdot \overrightarrow{M_2M_3} =0$. Additionally, from \textbf{(Step 2)}, we have  $\overrightarrow{M_2M_3}=2\overrightarrow{T_3T_5}$. Hence $\overrightarrow{OT_4} \cdot \overrightarrow{T_3T_5}=O$.
This is equivalent to  $\overrightarrow{OT_4} \cdot \overrightarrow{OT_5} = \overrightarrow{OT_4} \cdot \overrightarrow{OT_3}$.  Together with the fact that $|OT_5|=|OT_3|$, we have that $|T_4T_5|= |T_3T_4|$.  Similarly, we can show that $|T_4T_5| = |T_3T_5|$. 
Therefore, $T_4T_4T_5$ is an equilateral triangle.
\newline
\newline
\noindent \textbf{(Step 4)} $X$ is the right equilateral-triangular prism circumscribed by a sphere.
\newline
\newline
By Lemma \ref{combinatorial_triangular_prism_classification}, $AD$, $BE$, and $CF$ are parallel to each other or they concur at a point.
\newline

\noindent \textit{(Case 1)}: $AD$, $BE$, and $CF$ are parallel to each other.
\newline

We orient $X$ such that $AD$, $BE$, and $CF$ are parallel to the $z$-axis and $O$ is at the origin. Define $\pi: R^{3} \rightarrow R^{2}$ be the projection map from the whole Euclidean space to $xy$-plane.  Let $z(X)$ denote the $z$-component of any point $X$ in Euclidean space.

First, observe that the tangent planes of the sphere at points $T_3, T_4$ and $T_5$ are parallel to the z-axis. It follows that $z(T_3)=z(T_4)=z(T_5)=0$, so $T_3,T_4$ and $T_5$ lie on the $xy$-plane.  Then, by \textbf{(Step 3)}, we have that the centroid of $T_3T_4T_5$ is the origin $O$. It follows, by \textbf{(Step 2)}, that the centroid of $M_1M_2M_3$ is also the origin.

Because projection maps are linear, it preserves centroids.  Since the triangle $\pi(A)\pi(B)\pi(C)$ is equivalent to $M_1M_2M_3$, $\pi(T_1)$ is the origin $O$. Similarly, $\pi(T_2)$ is $O$.

Therefore, the $B_1$ and $B_2$ are perpendicular to the lines $AD$, $BE$, and $CF$. This implies that $B_1$, $B_2$, and $M_1M_2M_3$ are congruent to each other.  From \textbf{(Step 2)} and \textbf{(Step 3)}, the triangle $M_1M_2M_3$ is equilateral. Then $B_1$ and $B_2$ are also equilateral. Hence, $X$ is the unit-volume equilateral-triangular prism circumscribed about a sphere.
\newline

\noindent \textit{(Case 2)}: $AD$, $BE$, and $CF$ concur at a point.
\newline

We now orient $X$ such that $T_3T_4T_5$ is parallel to the $xy$-plane and $O$ is at the origin. Since $T_3T_4T_5$ is an equilateral triangle, the projection of $T_1$ to the $xy$-plane is the origin $O$. By \textbf{(Step 1)}, the midpoint of $T_1T_2$ also projects to the origin in $xy$-plane.

From the assumption of this case, $AD$, $BE$, and $CF$ are not parallel to the $z$-axis. Therefore, the plane containing $T_3T_4T_5$ does not contain the origin. Hence, the distances from the plane containing $T_3T_4T_5$ to $T_1$ and to $T_2$ are different.  Therefore, we deduce that $OT_1$ $\neq$ $OT_2$, a contradiction. It follows that this case is impossible.
\end{proof}

\begin{corollary}
\label{triprismtile}
The right equilateral-triangular prism circumscribed about a sphere, having base-length $4^{1/3}$ and height $4^{1/3}3^{-1/2}$, is the surface-area-minimizing 5-hedral tile.
\end{corollary}

\begin{proof} 
Since the prism is surface-area-minimizing by Theorem \ref{bestfivepoly} and is a tile, it gives the surface-area-minimizing tiling.
\end{proof}

\begin{remark}
\emph{Since equilateral triangles are the perimeter-minimizing polygons of 3 sides, Corollary \ref{triprismtile} also follows directly from Proposition \ref{montile}.}
\end{remark}



\bibliographystyle{abbrv}
\bibliography{main}

\bigskip
Whan Ghang, Massachusetts Institute of Technology
\newline
\textit{Email:} ghangh@MIT.edu
\newline
\newline
Zane Martin, Williams College 
\newline
\textit{Email:} zkm1@williams.edu
\newline
\newline
Steven Waruhiu, University of Chicago
\newline
\textit{Email:} waruhius@uchicago.edu

\end{document}